\documentclass[12pt]{iopart}
\expandafter\let\csname equation*\endcsname\relax
\expandafter\let\csname endequation*\endcsname\relax

\usepackage{amsmath,amsfonts,amssymb,amsthm}
\usepackage{graphicx}
\usepackage{color,colordvi}
\usepackage{hyperref}
\usepackage{enumerate}

\def\softd{{\leavevmode\setbox1=\hbox{d}%
\hbox to 1.05\wd1{d\kern-0.4ex{\char039}\hss}}}
\def\softt{{\leavevmode\setbox1=\hbox{t}%
\hbox to \wd1{t\kern-0.6ex{\char039}\hss}}}

\newcommand{\ee}{\mathrm{e}}
\newcommand{\D}{\mathrm{d}}
\newcommand{\C}{\mathbb{C}}
\newcommand{\N}{\mathbb{N}}
\newcommand{\R}{\mathbb{R}}
\newcommand{\CC}{\mathcal{C}}
\newcommand{\DD}{\mathcal{D}}
\newcommand{\OO}{\mathcal{O}}
\newcommand{\RR}{\mathcal{R}}
\newcommand{\VV}{\mathcal{V}}

\newtheorem{theorem}{Theorem}[section]
\newtheorem{lemma}{Lemma}[section]
\newtheorem{proposition}{Proposition}[section]

\newtheorem{remark}{Remark}[section]

\begin{document}

\title[Soft quantum waveguides]
{Spectral properties of soft quantum waveguides}

\author{Pavel Exner}
\address{Nuclear Physics Institute, Czech Academy of Sciences,
Hlavn\'{i} 130, \\ 25068 \v{R}e\v{z} near Prague, Czech Republic}
\address{Doppler Institute, Czech Technical University, B\v{r}ehov\'{a} 7, 11519 Prague, \\ Czech Republic}
\ead{exner@ujf.cas.cz}

\begin{abstract}
We consider a soft quantum waveguide described by a two-dimensional Schr\"odinger operators with an attractive potential in the form of a channel of a fixed profile built along an infinite smooth curve which is not straight but it is asymptotically straight in a suitable sense. Using Birman-Schwinger principle we show that the discrete spectrum of such an operator is nonempty if the potential well defining the channel profile is deep and narrow enough. Some related problems are also mentioned.
\end{abstract}


%
\vspace{2pc} \noindent{\it Keywords}: Quantum waveguides, Schr\"odinger operators, regular potentials, Birman-Schwinger principle, discrete spectrum

%
\submitto{\JPA}

%
%
%

\section{Introduction} 
\setcounter{equation}{0}

Properties of motion confined to regions infinitely extended in some direction, and of a constant `size' in the other, attracted over the years a lot of attention. The original motivation came from quantum mechanics where such systems appeared naturally as models of electrons in semiconductor wires, atoms in hollow optical fibers, and other waveguide-type systems \cite{EK15}, but related effects were also investigated, theoretically and experimentally, for instance in electromagnetism \cite{LCM99} or acoustics \cite{DP98}. The most common framework for such a waveguide analysis is based on Schr\"odinger operators in tubular regions with appropriate conditions at their boundary; for a recent survey of the existing results we refer to the monograph \cite{EK15}.

An alternative approach, for which names like `leaky quantum wires' or similar is used, is motivated by the fact that in actual quantum systems the confinement is rarely perfect, and tunneling between different parts of the structure is possible. It works with singular Schr\"odinger operators which can be formally written as $-\Delta-\alpha\delta(x-\Gamma)$ with $\alpha>0$, where the interaction support $\Gamma$ is a curve, a graph, or more generally, a complex of lower dimensionality \cite[Chap.~10]{EK15}. In this situation the particle can be found everywhere, but the states with the energy support in the negative part of the spectrum are localized in the vicinity of $\Gamma$.

The two model classes share some properties. A notable one among them is the existence of \emph{curvature-induced bound states}. If such a tube, or the set $\Gamma$ supporting the attractive interaction, is not straight but it is straight outside a bounded region, or it is at least asymptotically straight in a suitable sense, then the system exhibits localized states associated with discrete eigenvalues below the continuum, their number and location being determined by the geometry of the system.

One can note that both the described system types contain a degree of idealization. We have mentioned already that the tube walls may not be impenetrable. On the other hand, real guides we try to model are often thin but never infinitely thin. A more realistic way to describe this sort of physical constraints would be to use an attractive potential, but a regular one in the form of a ditch, or channel, of a finite width. This is the topic we are going to address in this paper. The first question that naturally arises is whether in such a situation again `bending can cause binding', that is, whether the Hamiltonian with a bent potential channel would have isolated eigenvalues.

We are going to discuss this problem in the simplest two-dimensional setting where the potential channel has a finite width. There are naturally many extensions to the question posed in this way, and we will mention them briefly in the concluding remarks. The tool we choose is the Birman-Schwinger method which has the advantage that it allows us employ the `straightening trick' known from the theory of hard-wall waveguides. This is not the only possibility, of course, one can think, for instance, of a variational method, but that could constitute the contents of another paper. Our main result, Theorem~\ref{thm:boundstate}, is a sufficient condition for the existence of curvature-induced bound states which shows, in particular, that such states exist provided the potential well determining the channel profile is deep and narrow enough.

Before proceeding, let us briefly describe the contents of the paper. The problem is stated in the next section where the assumption we use are formulated. In Section~\ref{s:essential} we combine the bracketing technique and Weyl criterion in order to determine the essential spectrum of the Schr\"odinger operators in question; it is vital to show that it is preserved under the geometric perturbations we are considering. Before coming to our main topic, we mention in Section~\ref{s:asympt} two simple asymptotic results establishing the discrete spectrum existence based on the known facts about the hard-wall waveguides and leaky wires. The main part of the paper are Sections~\ref{s:BS}--\ref{another} where we apply Birman-Schwinger method to our case, prove the main result, and discuss its consequences. We finish the paper with a survey of open questions related to the present problem, noting with pleasure a result on a similar problem in the three-dimensional setting that appeared very recently \cite{EKP20}.

\section{Statement of the problem} 
\setcounter{equation}{0}

Let $\Gamma$ be an infinite and smooth planar curve without self-intersections, naturally parametrized by its arc length $s$, that is, the graph of a function $\Gamma:\:\R\to\R^2$; with the abuse of notation we employ the same symbol for the map $\Gamma$ and for its range. Writing $\Gamma$ in the Cartesian coordinates we introduce the signed curvature $\gamma:\: \gamma(s)= (\dot\Gamma_2\ddot\Gamma_1 - \dot\Gamma_1\ddot\Gamma_2)(s)$ where the dot conventionally denotes the derivative with respect to $s$. The curve is supposed to satisfy the following assumptions:
 \begin{enumerate}[(a)]
 \setlength{\itemsep}{0pt}
\item $\Gamma$ is $C^2$-smooth so, in particular, $\gamma(s)$ makes sense, \label{assa}
\item $\gamma$ is either of compact support, $\mathrm{supp}\,\gamma \subset [-s_0,s_0]$ for some $s_0>0$, or $\Gamma$ is $C^4$-smooth and $\gamma(s)$ together with its first and second derivatives tend to zero as $|s|\to\infty$, \label{assb}
\item $|\Gamma(s)-\Gamma(s')|\to \infty$ holds as $|s-s'|\to\infty$. \label{assc}
 \end{enumerate}
The last assumption excludes U-shaped curves and their various modifications. The knowledge of $\gamma$ allows us to reconstruct the curve uniquely, up to Euclidean transformations: putting $\beta(s_2,s_1):= \int_{s_1}^{s_2} \gamma(s)\,\D s$, we have
 \begin{equation} \label{reconstruction}
 \Gamma(s) = \Big( x_1 + \int_{s_0}^s \cos\beta(s_1,s_0)\,\D s_1, x_2 - \int_{s_0}^s \sin\beta(s_1,s_0)\,\D s_1 \Big)
 \end{equation}
for some $s_0\in\R$ and $x=(x_1,x_2)\in\R^2$. Next we define the strip $\Omega^a$ in the plane, of halfwidth $a>0$, built over $\Gamma$ as
 $$ 
\Omega^a := \{ x\in\R^2:\:\mathrm{dist}(x,\Gamma) < a \},
 $$ 
in particular, $\Omega^a_0 := \R\times (-a,a)$ corresponds to a straight line for which we use the symbol $\Gamma_0$. We assume that
 \begin{enumerate}[(a)]
 \setcounter{enumi}{3}
 \setlength{\itemsep}{0pt}
\item $a\|\gamma\|_\infty<1$ holds for the strip halfwidth of $\Gamma$. \label{assd}
 \end{enumerate}
This ensures that the points of $\Omega^a$ can be uniquely parametrized by the arc length and the distance from $\Gamma$ as follows,
 \begin{equation} \label{strip}
x(s,u) = \big(\Gamma_1(s)-u\dot\Gamma_2(s), \Gamma_2(s)+u\dot\Gamma_1(s) \big),
 \end{equation}
which constitute a natural locally orthogonal system of coordinates on $\Omega^a$, $N(s) = (-\dot\Gamma_2(s),\dot\Gamma_1(s))$ being the unit normal vector to $\Gamma$ at the point $s$.

The main object of our interest are Schr\"odinger operators with an attractive potential supported in the strip $\Omega^a$. To introduce it we consider
 \begin{enumerate}[(a)]
 \setcounter{enumi}{4}
 \setlength{\itemsep}{0pt}
 \item a nonzero $V\ge 0$ from $L^\infty(\R)$ with $\mathrm{supp}\,V\subset [-a,a]$ \label{asse}
 \end{enumerate}
and define
 \begin{subequations}
 \begin{align}
 \tilde{V}:\: \Omega^a & \to\R_+,\quad \tilde{V}(x(s,u))=V(u), \label{potential} \\[.5em]
 H_{\Gamma,V} & = -\Delta-\tilde{V}(x); \label{Hamiltonian}
 \end{align}
 \end{subequations}
in view of assumption \eqref{asse} the operator domain is $D(-\Delta)=H^2(\R^2)$. It is also useful to introduce the comparison operator on $L^2(\R)$,
 \begin{equation} \label{transop}
h_V = -\partial_x^2-V(x)
 \end{equation}
with the domain $H^2(\R)$ which has in accordance with \eqref{asse} a nonempty and finite discrete spectrum such that
 \begin{equation} \label{infsph}
\epsilon_0 := \inf\sigma_\mathrm{disc}(h_V) = \inf\sigma(h_V)\in \big(-\|V\|_\infty,0\big).
 \end{equation}
Moreover, we know that $\epsilon_0$ is a simple eigenvalue and the associated eigenfunction $\phi_0\in H^2(\R)$ can be chosen strictly positive; we will use the same symbol for this function and its restriction to the interval $(-a,a)$. The relation \eqref{infsph} helps us to find the spectrum of $H_{\Gamma,V}$ in the situation when the generating curve is a straight line because in that case the variables separate and we have
 \begin{equation} \label{straightsp}
\sigma(H_{\Gamma_0,V}) = \sigma_\mathrm{ess}(H_{\Gamma_0,V}) = [\epsilon_0,\infty).
 \end{equation}

The question we address in this paper is about the spectrum of the operator $H_{\Gamma,V}$ in the situation where $\Gamma$ satisfies the assumptions \eqref{assa}--\eqref{asse} and \emph{is not straight}. An example of particular interest is the soft \emph{flat-bottom} waveguide referring to the function
 \begin{equation} \label{flatbottom}
V_{J,0}(u) = V_0\chi_J(u),\quad V_0>0,
 \end{equation}
where $\chi_J$ is the indicator function of an interval $J=[-a_1,a_2] \subset [-a_0.a_0]$.

\section{The essential spectrum}\label{s:essential}
\setcounter{equation}{0}

If the potential ditch is straight outside a compact, or at least asymptotically straight in the sense of \eqref{assb}, the essential spectrum is preserved.

\begin{proposition} \label{prop:essential}
Under assumptions \eqref{assa}--\eqref{asse} we have $\sigma_\mathrm{ess}(H_{\Gamma,V}) = [\epsilon_0,\infty)$.
\end{proposition}
\begin{proof}
If $\Gamma$ is straight outside a compact we can divide the plane into four regions. The first two of them is a pair of \emph{disjoint} halfstrips $\Sigma_\pm=\{x(s,u):\: \pm s>s_1, |u|<u_1\}$, the rest consists of a compact set $\Sigma_\mathrm{c}$ containing $\mathrm{supp}\,V \setminus(\Sigma_+ \cup \Sigma_-)$ and its complement to $\R^2 \setminus(\Sigma_+ \cup \Sigma_-)$. We estimate $H_{\Gamma,V}$ from below by imposing additional Neumann condition at the boundaries between the four regions. The compact one does not contribute to the essential spectrum, hence we have to inspect the spectral thresholds of the other three. The operator part referring to the last named one is positive, while the operators corresponding to the two halfstrips have separated variables, and consequently, their essential spectra start at $\epsilon(u_1):=\inf\sigma(h^\mathrm{N}_V(u_1))$, where $h^\mathrm{N}_V(u_1)$ is the operator \eqref{transop} restricted to the interval $(-u_1,u_1)$ with Neumann boundary conditions. Furthermore, in analogy with \cite{DH93} one can check that $\epsilon(u_1)\to\epsilon_0$ as $u_1\to\infty$, in fact exponentially fast, because $V(x)=0$ outside $(-a,a)$ and $\phi_0$ is exponentially decreasing there. Noting finally that in view of assumption~\eqref{assc} the halfstrips width $2u_1$ can be chosen arbitrarily large by picking $s_1$ large enough, we can conclude that
 \begin{equation} \label{infspess}
\inf\sigma_\mathrm{ess}(H_{\Gamma,V}) = \epsilon_0.
 \end{equation}

If $\Gamma$ satisfies the other part of assumption \eqref{assb} we replace the halfstrips $\Sigma_\pm$ by families of pairwise disjoint regions $\Sigma^{(j)}_\pm := \{ x(s,u):\: \pm s\in(s_{2j-1},s_{2j}), u\in(-u_j,u_j)\},\, j\in\N$, determined by increasing sequences $\{s_j\}_{j=1}^\infty$ and $\{u_j\}_{j=1}^\infty$ of positive numbers, and impose again Neumann conditions at their boundaries. Using the parametrization \eqref{strip} we can pass from the corresponding Neumann restrictions of $h_V$ to unitarily equivalent operators $H^{(j)}_\pm$ on $L^2(\Sigma^{(j)}_{0,\pm})$, where $\Sigma^{(j)}_{0,+}:= (s_{2j-1},s_{2j})\times (-u_j,u_j)$ and $\Sigma^{(j)}_{0,-}$ is defined analogously, in the way described in \eqref{straightening} below. According to \cite[Sec.~1.1]{EK15} they are of the form
 \begin{subequations}
 \begin{align} 
 H^{(j)}_\pm = & \;\, h^\mathrm{N}_V(u_1) \otimes (-\partial_s^2)_\mathrm{N} + V_\gamma(s,u), \nonumber \\[.5em]
 V_\gamma(s,u):= & -\frac{\gamma(s)^2}{4(1+u\gamma(s))^2} + \frac{u\ddot\gamma(s)}{2(1+u\gamma(s))^3} - \frac54\, \frac{u^2\dot\gamma(s)^2}{(1+u\gamma(s))^4}. \nonumber 
 \end{align}
 \end{subequations}
Since $\gamma(s)\to 0$ as $|s|\to\infty$, in view of assumptions \eqref{assc} and \eqref{assd} the rectangles $\Sigma^{(j)}_{0,\pm}$ can be made arbitrarily wide in the $u$ variable by choosing large enough $s_{2j-1}$. The spectrum of each of the operators $H^{(j)}_\pm$ is discrete, of course, it is the accumulation point of their principal eigenvalues which determines the threshold of $\sigma_\mathrm{ess}(H_{\Gamma,V})$. Since the ground state of $(-\partial_s^2)_\mathrm{N}$ is zero and the eigenvalues of $H^{(j)}_\pm$ differ from those of $h^\mathrm{N}_V(u_1) \otimes (-\partial_s^2)_\mathrm{N}$ at most by $\|V_\gamma \upharpoonright \Sigma^{(j)}_{0,\pm}\|_\infty$ which tends to zero as $j\to\infty$, we arrive at \eqref{infspess} again.

To prove that there are no spectral gaps above $\epsilon_0$, we use Weyl criterion. To construct a suitable sequence, we consider functions $v,w\in C_0^\infty(\R)$ with the supports in $[-1,1]$ such that their norms are $\|v\|=\|w\|=1$ and $v(s)=w(s)=1$ holds in the vicinity of zero, and put
 \begin{equation} \label{weylseq}
\psi(s,u) = \sqrt{\mu\nu}\, v(\mu(s-s_0))\,w(\nu u)\,\phi_0(u)\,\ee^{iks}
 \end{equation}
for $\mu,\,\nu>0$, a fixed $k\in\R$, and some $s_0$. The conclusions we are going to make do not depend on $s_0$, so we can put $s_0=0$. Using the fact that $\phi_0$ is the eigenfunction of $h_V$ corresponding to the eigenvalue $\epsilon_0$ we find that
 \begin{align*} \label{weyleval}
 (-\Delta-V-\epsilon_0-k^2)\psi(s,u) & = \sqrt{\mu\nu} \big[ \big(-\mu^2 v''(\mu s)^2 - 2ik\mu v'(\mu s)\big) w(\nu u)\phi_0(u) \\[.3em]
 & + v(\mu s) \big( -\nu^2 w(\nu u)\phi_0(u) - \nu w'(\nu u)\phi'_0(u) \big) \big]\,\ee^{iks}
 \end{align*}
The right-hand side is a sum of four terms, $f_1+f_2+f_3+f_4$. Let us estimate their norms. After a simple change of variables we get
 $$ 
 \|f_1\|^2 = \int_{-1}^1 \int_{-1}^1 \mu^4 v''(\xi)^2 w(\eta)^2 \phi_0\big(\textstyle{\frac{u}{\nu}}\big)\,\D\xi\D\eta = \mu^4 \|w''\|^2 \phi_0(0) \big(1+\OO(\nu) \big)
 $$ 
so that $\|f_1\|=\OO(\mu^2)$, in a similar way we find $\|f_2\|=\OO(\mu)$, $\|f_3\|=\OO(\nu^2)$, and $\|f_4\|=\OO(\nu)$, and therefore $\| (-\Delta-V-\epsilon_0-k^2)\psi\| \to 0$ as $\mu,\,\nu\to 0$. To prove that any number larger than $\epsilon_0$ belongs to $\sigma_\mathrm{ess}(H_{\Gamma,V})$ it is thus sufficient to find a family of increasing regions threaded by the curve that support functions $\psi_j$ of the form \eqref{weylseq}. If $\Gamma$ is straight outside a compact we use the regions $\Sigma^{(j)}_{0,+}$ and put $s_0=\frac12(s_{2j-1}+s_{2j})$ in \eqref{weylseq}. In contrast to the first part of the proof, in addition to $u_j\to\infty$ it is also important to have $s_{2j}-s_{2j-1}\to\infty$ as $j\to\infty$, which is possible due to assumption \eqref{assc}. Then it is enough to choose $\mu_j=2c(s_{2j}-s_{2j-1})^{-1}$ and $\nu_j=cu_j^{-1}$ for some $c\in(0,1)$.

For $\Gamma$ satisfying the other part of assumption \eqref{assb} we replace the rectangles $\Sigma^{(j)}_{0,+}$ by the `bent rectangles' $\Sigma^{(j)}_+$, assumptions \eqref{assc} and \eqref{assd} again allow us to choose a disjoint family of them expanding in both the longitudinal and transversal directions. Using once more the straightening transformation we find that the norms of the corresponding Weyl approximants, $\| (-\Delta-\tilde{V}-\epsilon_0-k^2)\psi_j\|$, differ from the above estimate at most by $\|V_\gamma \upharpoonright \Sigma^{(j)}_{0,\pm}\|_\infty$, and since this quantity vanishes as we follow the curve to infinity, the proof is complete.
\end{proof}

\section{Asymptotic results}\label{s:asympt}
\setcounter{equation}{0}

The other types of quantum waveguides, the hard-wall ones \cite{EK15} and the leaky wires, described by singular Schr\"odinger operators, are not only much better understood than operators of the type \eqref{Hamiltonian}, but they represent in a sense extreme cases of such systems. This makes it possible to prove some sufficient conditions for the existence of discrete spectrum. One comes from the approximation result proven in \cite{EI01}, and in greater generality in \cite{BEHL17}. To state it, we consider the family of potentials
 $$ 
 V_\varepsilon:\: V_\varepsilon(u) = \textstyle{\frac{1}{\varepsilon}} V\big(\textstyle{\frac{u}{\varepsilon}}\big)
 $$ 
obtained by scaling of a given $V$ satisfying assumption \eqref{asse}.

\begin{proposition} \label{prop:scaling}
Consider a non-straight $C^2$-smooth curve $\Gamma:\:\R\to\R^2$ such that $|\Gamma(s)-\Gamma(s')|<c|s-s'|$ holds for some $c\in(0,1)$. If the support of its signed curvature $\gamma$ is noncompact, assume, in addition to \eqref{assb}, that $\gamma(s) = \OO(|s|^{-\beta})$ with some $\beta>\frac54$ as $|s|\to\infty$. Then $\sigma_\mathrm{disc}(H_{\Gamma,V_\varepsilon}) \ne \emptyset$ holds for all $\varepsilon$ small enough.
\end{proposition}
\begin{proof}
By the results of \cite{EI01, BEHL17} the operators $H_{\Gamma,V_\varepsilon}$ converge as $\varepsilon\to 0$ in the norm-resolvent sense to the singular Schr\"odinger operator $H_{\Gamma,\alpha}$ formally written as $-\Delta-\alpha\delta(x-\Gamma)$ with $\alpha:= \int_{-a}^a V(u)\,\D u$, which is the unique self-adjoint operator associated with the quadratic form
 $$ 
 q_{\Gamma,\alpha}[\psi] = \|\nabla\psi\|^2_{L^2(\R)} - \alpha \|\psi\|^2_{L^2(\Gamma)},
 $$ 
closed and below bounded on $H^1(\R^2)$. Consequently, the spectrum of $H_{\Gamma,V_\varepsilon}$ converges in the set sense to that of $H_{\Gamma,\alpha}$ as $\varepsilon\to 0$. In particular, it is not difficult to check directly that the essential spectrum threshold $\epsilon_0(\varepsilon)$ of $H_{\Gamma,V_\varepsilon}$ converges to the the essential spectrum threshold $-\frac14\alpha^2$ of $H_{\Gamma,\alpha}$.

The assumptions \eqref{assa}, \eqref{assb}, and \eqref{asse} are satisfied, the inequality $|\Gamma(s)-\Gamma(s')|<c|s-s'|$ implies further the validity of \eqref{assc}. Next we note that the support of $V_\varepsilon$ is contained in $[-\varepsilon a,\varepsilon a]$, hence \eqref{assd} is valid for sufficiently small values of $\varepsilon$. At the same time, the additional decay requirement imposed on $\gamma$ means by Remark~5.6 of \cite{EI01} that the assumptions of Theorem~5.2 of the said paper are satisfied, and consequently, the operator $H_{\Gamma,\alpha}$ has at least one isolated eigenvalue below the the essential spectrum threshold. From the convergence result indicated above, it then follows that the same is true for $H_{\Gamma,V_\varepsilon}$ with $\varepsilon$ small enough.
\end{proof}

The most common quantum waveguide model works with the particle confined to a tubular region with hard walls; if no other forces are involved, the Hamiltonian is (a multiple of) the appropriate Dirichlet Laplacian. In the two-dimensional situation where the region is a strip of a fixed width in the plane, the known results can again be used for comparison with our present problem. Let us recall the heuristic statement that the Dirichlet condition corresponds to `infinitely high potential wall', which can be given a mathematically rigorous meaning, cf. \cite[Sec.~4.2.3]{DK05} or \cite[Sec.~21]{Si05}.

\begin{proposition} \label{prop:hardwall}
Suppose that $\Gamma$ is not straight and assumptions \eqref{assa}--\eqref{assd} are satisfied, then the operator $H_{\Gamma,V_{J,0}}$ referring to the potential \eqref{flatbottom} has nonempty discrete spectrum for all $V_0$ large enough.
\end{proposition}
\begin{proof}
$\{H_{\Gamma,V_{J,0}}:\,V_0\ge 0\}$ is clearly a holomorphic family of type (A) in the sense of Kato, and moreover, it is monotonous with respect to $V_0$. The same is true for operators $G_{\Gamma,V_{J,0}}:= H_{\Gamma,V_{J,0}}+V_0$ which are positive and form an increasing family, $G_{\Gamma,V_{J,0}} \ge G_{\Gamma,V'_{J,0}}$ for $V_0>V'_0$. Thus their eigenvalues $\lambda_j(V_0)$ are continuous functions of $V_0$, increasing by the minimax principle, and as such that they have limits as $V_0\to\infty$; the same also applies to the threshold of their essential spectrum. The limiting values refer to the corresponding spectral quantities of the (negative) Dirichlet Laplacian of $\Omega^J$, the support of the function \eqref{flatbottom}. Under the hypotheses made, the assumptions of Theorem~1.1 and its corollary in \cite{EK15} are fulfilled, hence the said limiting operator has a nonempty discrete spectrum below the continuum which starts at $\pi^2|J|^{-2},\, |J|=a_1+a_2$. It follows that for all sufficiently large $V_0$ we have $\sigma_\mathrm{disc}\big(G_{\Gamma,V_{J,0}}\big) \ne\emptyset$, and the same is true, of course, for the shifted operators $H_{\Gamma,V_{J,0}}$.
\end{proof}

\begin{remark} 
{\rm We note that the assumptions under which the used spectral properties of the limiting operators have been derived are in no way optimal. This concerns, for instance, the existence of a wedge separating the curve ends in Proposition~\ref{prop:scaling} or the flat bottom of the potential channel in Proposition~\ref{prop:hardwall}. There is, no doubt, a room for mathematical activity here.}
\end{remark}

\section{Birman-Schwinger analysis}\label{s:BS}
\setcounter{equation}{0}

Given a function $V$ and $z\in\C\setminus\R_+$ we consider the operator
 \begin{equation} \label{BSop}
K_{\Gamma,V}(z) := \tilde{V}^{1/2} (-\Delta-z)^{-1} \tilde{V}^{1/2}
 \end{equation}
with $\tilde{V}$ given by \eqref{potential}; we are particularly interested in the negative values of the spectral parameter, $z=-\kappa^2$ with $\kappa>0$. In view of assumption \eqref{asse} it is a bounded operator, positive for $z=-\kappa^2$, which we can regard as a map $L^2(\Omega^a) \to L^2(\Omega^a)$. By Birman-Schwinger principle this operator can be used to determine the discrete spectrum of $H_{\Gamma,V}$.
\begin{proposition} \label{prop:BS}
$z\in\sigma_\mathrm{disc}(H_{\Gamma,V})$ holds if and only if $\,1\in\sigma_\mathrm{disc}(K_{\Gamma,V}(z))$. The function $\kappa \mapsto K_{\Gamma,V}(-\kappa^2)$ is continuous and decreasing in $(0,\infty)$, tending to zero in the norm topology, that is, $\|K_{\Gamma,V}(-\kappa^2)\| \to 0$ holds as $\kappa\to\infty$.
\end{proposition}
\begin{proof}
The first claim is a particular case of a more general and commonly known result, see, e.g., \cite{BGRS97}. The continuity follows from the functional calculus and  we have
 $$ 
\frac{\D}{\D\kappa} (\psi,\tilde{V}^{1/2} (-\Delta+\kappa^2)^{-1}\, \tilde{V}^{1/2}\psi) = -2\kappa (\psi,\tilde{V}^{1/2} (-\Delta+\kappa^2)^{-2}\, \tilde{V}^{1/2}\psi) < 0
 $$ 
for any $\psi\in L^2(\Omega^a)$ with $\tilde{V}^{1/2}\psi \ne 0$ which proves the monotonicity. Finally, the simple estimate $\|K_{\Gamma,V}(-\kappa^2)\| \le \kappa^{-2} \|V\|_\infty$ concludes the proof.
\end{proof}

Note also that if $g$ is an eigenfunction of the operator \eqref{BSop} with eigenvalue one, the corresponding eigenfunction of $H_{\Gamma,V}$ is given by
 \begin{equation} \label{reconst}
 \phi(x) = \int_{\mathrm{supp}\,\tilde{V}} G_\kappa(x,x')\, \tilde{V}(x')^{1/2}g(x')\, \D x',
 \end{equation}
where $G_\kappa$ is the integral kernel of $(-\Delta+\kappa^2)^{-1}$.

Using the knowledge of the Laplacian resolvent we can write the action of $K_{\Gamma,V}(z)$ explicitly, in particular, for $z=-\kappa^2$ with $\kappa>0$ it is an integral operator with the kernel
 $$ 
K_{\Gamma,V}(x,x';-\kappa^2) = \frac{1}{2\pi}\,\tilde{V}^{1/2}(x) K_0(\kappa|x-x'|) \tilde{V}^{1/2}(x'),
 $$ 
where $K_0$ is the Macdonald function, mapping $L^2(\Omega^a)$ to itself.

In order to treat the geometry of $\Gamma$ as a perturbation it is useful to pass to the `straightened' strip in the way used in treatment of the `hard-wall' waveguides \cite[Sec.~1.1]{EK15}. The first step is the natural coordinate change, passing from the Cartesian coordinates in the plane to $s,u$, which amounts to a unitary map $L^2(\Omega^a) \to L^2(\Omega^a_0, (1+u\gamma(s))^{1/2}\D s\D u)$. To get rid of the Jacobian, we use the unitary operator
 \begin{equation} \label{straightening}
L^2(\Omega^a) \to L^2(\Omega^a_0), \quad (U\psi)(s,u) = (1+u\gamma(s))^{1/2} \psi(x(s,u)).
 \end{equation}
Using it we pass from the Birman-Schwinger operator $K_{\Gamma,V}(-\kappa^2)$ to the unitarily equi\-valent one, $\RR^\kappa_{\Gamma,V} := U K_{\Gamma,V}(-\kappa^2) U^{-1}$, which is an integral operator on $L^2(\Omega^a_0)$ with the kernel
 \begin{equation} \label{uniteq}
\RR^\kappa_{\Gamma,V}(s,u;s',u') = \frac{1}{2\pi}\,W(s,u)^{1/2} K_0(\kappa|x-x'|)\, W(s',u')^{1/2},
 \end{equation}
where $x=x(s,u)$, $x'=x(s',u')$, and the modified potential is
 $$ 
W(s,u) := (1+u\gamma(s))\,V(u).
 $$ 

\section{The straight case} 
\setcounter{equation}{0}

The Birman-Schwinger operator corresponding to the straight potential ditch has the kernel
 \begin{equation} \label{straightker}
\RR^\kappa_{\Gamma_0,V}(s,u;s',u') = \frac{1}{2\pi}\,V(u)^{1/2} K_0(\kappa|x_0-x'_0|)\, V(u')^{1/2},
 \end{equation}
where $|x_0-x'_0| = \big[ (s-s')^2 + (u-u')^2\big]^{1/2}$. To find its spectrum we notice that, with respect to the longitudinal variable $s$, it acts as a convolution. The unitarily equivalent operator obtained by means  of the Fourier-Plancherel operator $F$ on $L^2(\R)$, has thus the form of a direct integral,
 \begin{equation} \label{dirint}
 (F\otimes I) \RR^\kappa_{\Gamma_0,V} (F\otimes I)^{-1} = \int^\oplus_\R \RR^\kappa_{\Gamma_0,V}(p)\, \D p,
 \end{equation}
where the fibers are integral operator on $L^2(-a,a)$  with the kernels
\begin{equation} \label{1dBSkernel}
 \RR^\kappa_{\Gamma_0,V}(u,u';p) = V(u)^{1/2}\, \frac{\ee^{-\sqrt{\kappa^2+p^2}|u-u'|}}{2\sqrt{\kappa^2+p^2}}\, V(u')^{1/2},
\end{equation}
because by \cite[6.726.4]{GR07} and \cite[10.2.17]{AS72} we have
 $$ 
 \frac{1}{2\pi}\, \int_\R K_0\big(\kappa\sqrt{\xi^2+|u-u'|^2}\big)\, \ee^{ip\xi}\, \D\xi = \frac{\ee^{-\sqrt{\kappa^2+p^2}|u-u'|}}{2\sqrt{\kappa^2+p^2}}.
 $$ 
However, in view of \eqref{1dBSkernel} $\RR^\kappa_{\Gamma_0,V}(p)$ is nothing but the Birman-Schwinger operator associated with \eqref{transop} referring to the spectral parameter $z=-(\kappa^2+p^2)$. By assumption, $\epsilon_0$ is the smallest eigenvalue of $h_V$, and consequently, by Proposition~\ref{prop:BS} in combination with \eqref{dirint}, the number $-\kappa^2=\epsilon_0+p^2$ belongs to the spectrum of $H_{\Gamma_0,V}$ for any $p\in\R$, in accordance with \eqref{straightsp}.

At the same time, the operator with kernel \eqref{straightker} satisfies
 $$ 
 \sup \sigma(\RR^{\kappa_0}_{\Gamma_0,V}) = 1,
 $$ 
where $\kappa_0= \sqrt{-\epsilon_0}$, because, was the left-hand side larger than one, by Proposition~\ref{prop:BS} there would exist a $\tilde\kappa > \kappa_0$ such that $1\in \sigma(\RR^{\tilde\kappa}_{\Gamma_0,V})$, and consequently, we would have $-\tilde\kappa^2 \in \sigma(H_{\Gamma_0,V})$, however, this contradicts relation \eqref{straightsp}.

We also note that one can relate the eigenfunction $\phi_0$ of $h_V$ to the eigenfunction $g_0$ of $\RR^\kappa_{\Gamma_0,V}(0)$ corresponding to the unit eigenvalue. On the one hand, we have
 \begin{equation} \label{g_0}
 g_0(u) = V^{1/2}(u)\phi_0(u),
 \end{equation}
on the other hand, $\phi_0$ can be expressed in the way analogous to \eqref{reconst} and, \emph{mutatis mutandis}, one can write the generalized eigenfunction associated with the bottom of $\sigma(H_{\Gamma_0,V})$ as
 $$ 
 f_0(s,u) = \phi_0(u) = \int_{-a}^a \frac{\ee^{-\kappa_0|u-u'|}}{2\kappa_0}\,V(u')^{1/2}\,g_0(u')\, \D u'.
 $$ 

\section{Existence of bound states} 
\setcounter{equation}{0}

Our aim is now to use the Birman-Schwinger method to derive a condition which would ensure the existence of curvature-induced bound states, that is, a discrete spectrum of operator $H_{\Gamma,V}$.

\begin{theorem} \label{thm:boundstate}
Let assumptions \eqref{assa}--\eqref{asse} be valid and set
 \begin{align*}
\CC^\kappa_{\Gamma,V}(s,u;s',u') = & \frac{1}{2\pi}\,\phi_0(u)V(u)\,\big[((1+u\gamma(s))^{1/2}\,K_0(\kappa|x(s,u)-x(s',u')|)\,(1+u'\gamma(s'))^{1/2} \nonumber \\[.3em] & -  K_0(\kappa|x_0(s,u)- x_0(s',u')|)\,\big]\, V(u')\phi_0(u')
 \end{align*}
for all $(s,u),\,(s',u') \in \Omega^a_0$, then we have $\sigma_\mathrm{disc}(H_{\Gamma,V}) \ne \emptyset$ provided
 \begin{equation} \label{BSsufficient}
 \int_{\R^2} \D s\D s' \int_{-a}^a \int_{-a}^a \D u\D u'\,\CC^{\kappa_0}_{\Gamma,V}(s,u;s',u') > 0.
 \end{equation}
holds for $\kappa_0= \sqrt{-\epsilon_0}$.
\end{theorem}
\begin{proof}
The main idea is to treat the geometry of the system, translated by \eqref{straightening} into the coefficients of the operator, as a perturbation of the straight case. By Proposition~\ref{prop:essential} the essential spectrum is preserved, hence in order to demonstrate that the spectrum of $H_{\Gamma,V}$ has at least one eigenvalue below $\epsilon_0$, it is in view of Proposition~\ref{prop:BS} sufficient to find a function $\psi_\eta\in L^2(\Omega^a_0)$ such that
 \begin{equation} \label{variation}
 (\psi, \RR^{\kappa_0}_{\Gamma,V}\psi) - \|\psi\|^2 > 0
 \end{equation}
As usual a natural way to construct such a trial function is to combine the generalized eigenfunction, associated with the edge of the spectrum, with a mollifier which makes it an element of the Hilbert space. Let us first inspect the effect of the mollifier in the straight case, $\Gamma=\Gamma_0$.
\begin{lemma} \label{l:straight}
Let $\psi_\eta\in L^2(\Omega^a_0)$ be of the form $\psi_\eta(s,u) = h_\eta(s) g_0(u)$, where $h_\eta(s)=h(\eta s)$ with a function $h\in C_0^\infty(\R)$ such that $h(s)=1$ in the vicinity of $s=0$. Then
 $$ 
 (\psi_\eta,  \RR^{\kappa_0}_{\Gamma_0,V}\psi_\eta) - \|\psi_\eta\|^2 = \OO(\eta)
 $$ 
holds as $\eta\to 0$.
\end{lemma}
\begin{proof}
Since $g_0$ is by assumption the eigenfuction of $\RR^{\kappa_0}_{\Gamma_0,V}(0)$, we can rewrite the second term on the left-hand side as $-\|h_\eta\|^2 (g_0, \RR^{\kappa_0}_{\Gamma_0,V}(0) g_0)$. Using relations \eqref{straightker}--\eqref{1dBSkernel} we cast the expression to be estimated into the form
 \begin{align*}
 \int_{-a}^a \int_{-a}^a & g_0(u)V(u)^{1/2} \bigg[\int_\R |\hat h_\eta(p)|^2\, \frac{\ee^{-\sqrt{\kappa_0^2+p^2}|u-u'|}}{2\sqrt{\kappa_0^2+p^2}}\,\D p
 - \|h_\eta\|^2 \frac{\ee^{-\kappa_0|u-u'|}}{2\kappa_0} \bigg] \\[.3em] & \times V(u')^{1/2} g_0(u')\,\D u\D u',
 \end{align*}
and since the functions $g_0$ and $V$ are bounded, it is enough to check that
 $$ 
 \int_\R |\hat h_\eta(p)|^2\, \frac{\ee^{-\sqrt{\kappa_0^2+p^2}|u-u'|}}{2\sqrt{\kappa_0^2+p^2}}\,\D p
 - \|h_\eta\|^2 \frac{\ee^{-\kappa_0|u-u'|}}{2\kappa_0} = \OO(\eta)
 $$ 
holds as $\eta\to 0$. We have $\hat h_\eta(p) = \frac1\eta\, \hat h\big( \frac{p}{\eta}\big)$, hence, changing the integration variable and using the mean value theorem, we can rewrite the first term on the left-hand side of this expression as
 $$ 
 \frac1\eta\, \int_\R |\hat h(\zeta)|^2\, \frac{\ee^{-\sqrt{\kappa_0^2+\eta^2\zeta^2}|u-u'|}}{2\sqrt{\kappa_0^2+\eta^2\zeta^2}}\,\D\zeta
 = \frac1\eta\,  \Big( \frac{\ee^{-\kappa_0|u-u'|}}{2\kappa_0} + \OO(\eta^2) \Big)\|h\|^2,
 $$ 
and using further the relation $\|h_\eta\|^2 = \frac1\eta\,\|h\|^2$ we arrive at the result.
\end{proof}

Now we can return to the proof of Theorem~\ref{thm:boundstate}. Consider the difference of the Birman-Schwinger operators
 \begin{equation} \label{difference}
 \DD^\kappa_{\Gamma,V} := \RR^\kappa_{\Gamma,V} - \RR^\kappa_{\Gamma_0,V}
 \end{equation}
which is by \eqref{uniteq} and \eqref{straightker} an integral operator on $L^2(\Omega^a_0)$ with the kernel
 \begin{align}
\DD^\kappa_{\Gamma,V}(s,u;s',u') = & \frac{1}{2\pi}\,\Big(W(s,u)^{1/2} K_0(\kappa|x(s,u)-x(s',u')|)\, W(s',u')^{1/2} \nonumber \\[.3em] & - V(u)^{1/2} K_0(\kappa|x_0(s,u)- x_0(s',u')|)\, V(u')^{1/2} \Big) \nonumber 
 \end{align}
In view of \eqref{variation}, \eqref{difference}, and Lemmma~\ref{l:straight} we would have $\sup \sigma(\RR^{\kappa_0}_{\Gamma_0,V}) > 1$ provided that
 $$ 
 \lim_{\eta\to 0} (\psi_\eta, \DD^{\kappa_0}_{\Gamma,V} \psi_\eta) > 0,
 $$ 
and using the choice of the function $h_\eta$, in combination with the dominant convergence theorem, we find that this is equivalent to
 $$ 
 \int_{\R^2} \D s\D s' \int_{-a}^a \int_{-a}^a \D u\D u'\, g_0(u)\, \DD^{\kappa_0}_{\Gamma,V}(s,u;s',u')\, g_0(u') > 0,
 $$ 
however, in view of \eqref{g_0} the integrated function is nothing but $\CC^{\kappa_0}_{\Gamma,V}$, hence the last inequality yields the condition \eqref{BSsufficient}.
\end{proof}

In contrast to the results of Section~\ref{s:asympt}, the sufficient condition \eqref{BSsufficient} is of a quantitative nature. Indeed, given $\Gamma$ and $V$, we are able to find $\kappa_0$ and the respective integral can be evaluated. We have already mentioned that the argument of the Macdonald function in the straight case is
 $$ 
 |x_0(s,u)-x_0(s',u')| = \big[ (s-s')^2 + (u-u')^2\big]^{1/2},
 $$ 
while for a general curve we find using \eqref{reconstruction} with $s'=s_0$ and \eqref{strip} that the squared distance is given by the formula
 \begin{align} 
 & |x(s,u)-x(s',u')|^2 \nonumber \\ &= |\Gamma(s)-\Gamma(s')|^2 + u^2 + u'^2 -2uu'\cos\beta(s,s') + 2(u\cos\beta(s,s')-u')\int_{s'}^s \sin\beta(\xi,s')\,\D\xi, \nonumber
 \end{align}
where the first term on the right-hand side equals
 $$ 
 |\Gamma(s)-\Gamma(s')|^2 = \int_{s'}^s\, \int_{s'}^s \cos\beta(\xi,\xi')\, \D\xi\,\D\xi'.
 $$ 

\section{Another existence result}\label{another}
\setcounter{equation}{0}

While we know the explicit form of the integral in \eqref{BSsufficient}, its evaluation may be complicated. Even without it, however, we can use Theorem~\ref{thm:boundstate} to make conclusions about the existence of a discrete spectrum. As an example, let us show a result which may be regarded as a counterpart to Proposition~\ref{prop:scaling}:

\begin{proposition} 
Let $\VV_{\epsilon_0}$ be the family of potentials $V$ satisfying assumption \eqref{asse} and such that $\inf\sigma(h_V)=\epsilon_0$. Then to any $\epsilon_0>0$ there exists an $a_0=a_0(\epsilon_0)$ such that $\sigma_\mathrm{disc}(H_{\Gamma,V}) \ne \emptyset$ holds for all $V\in\VV_{\epsilon_0}$ with $\mathrm{supp}\,V\subset[-a_0,a_0]$.
\end{proposition}
\begin{proof}
With the properties of the functions involved in mind, the integration in \eqref{BSsufficient} can be performed in any order; we rewrite thus the expression in question as
 \begin{equation} \label{reformulated}
 \frac{1}{2\pi}\,\int_{-a}^a \int_{-a}^a \phi_0(u)V(u)\,F(u,u')\, V(u')\phi_0(u)\,\D u\D u',
 \end{equation}
where
 \begin{align*}
F(u,u') := & \int_{\R^2} \big[((1+u\gamma(s))^{1/2}\,K_0(\kappa_0|x(s,u)-x(s',u')|)\,(1+u'\gamma(s'))^{1/2} \\[.3em] & \qquad -  K_0(\kappa_0|x_0(s,u)- x_0(s',u')|)\,\big]\, \D s\D s'.
 \end{align*}
The function $F(\cdot,\cdot)$ is well defined as long as assumption \eqref{assd} is satisfied and continuous. Furthermore, we have
 \begin{equation} \label{onthecurve}
F(0,0) = \int_{\R^2} \big[K_0(\kappa_0|\Gamma(s)-\Gamma(s')|) -  K_0(\kappa_0|s-s'|)\big]\, \D s\D s' > 0.
 \end{equation}
To see that this is the case, recall that $\Gamma$ is parametrized by the arc length so that $|\Gamma(s)-\Gamma(s')| \le |s-s'|$, and because the curve is not straight by assumption, there is an open set on which the inequality is sharp; this yields the inequality \eqref{onthecurve} since $K_0(\cdot)$ is decreasing in $(0,\infty)$. The continuity implies the existence of a neighborhood $(-a_0,a_0) \times (-a_0,a_0)$ of $(0,0)$ on which $F(u,u')$ is positive, and that in view of \eqref{reformulated} in combination with the positivity of $\phi_0V$ concludes the proof.
\end{proof}

\section{Concluding remarks} 
\setcounter{equation}{0}

In contrast to the cases of hard-wall waveguides and leaky wires, where \emph{any} bending gives rise -- under appropriate asymptotic straightness requirement -- to a nonempty discrete spectrum, the present analysis does not yield such a universal result. The question whether all bent soft waveguides, in particular the wide and shallow ones, do bind, remains open, and it is possible that the Birman-Schwinger technique is not an optimal tool to address it. True, it produced the universal result for leaky wires \cite{EI01} but it might have been a lucky coincidence. Recall also that there are other geometric perturbations for which the question about weakly coupled bound states is more subtle, see \cite{EV97} and notes to Sec.~6.4. in \cite{EK15}.

Apart from this question, the above discussion brings to mind numerous modifications and extensions of the present problem for which the \emph{soft waveguide program} might be an appropriate name:
\begin{enumerate}[(i)]
\setlength{\itemsep}{-1.5pt}
\item Potential channels in which the generating curve \emph{is not smooth}, as well as related \emph{polygonal potential channels}. The simplest example is a flat-bottom channel in the form of a broken strip. One expects that it will have a discrete spectrum for any channel depth provided the angle is sharp enough, while the answer in the general case is not clear. Note that for a cross-shaped potential channel, related to the question (iv) below, such a discrete spectrum exists independently of the ditch depth as it was shown by a variational method in \cite{EKP18}\footnote{The aim of that paper was to investigate two interacting particles on a halfline, but by symmetry considerations this is equivalent a single particle in crossed potential channels. Let us add that it is not for the first time when a link was made between a two-body system and a two-dimensional waveguide, cf. the caricature model of meson confinement in \cite{LLM86}.}. Unfortunately, this fact alone does not allow us to make conclusions about a single broken channel.
\item Tubular potential channels \emph{in three dimensions}. In view of \cite[Thms~1.3. and 10.3]{EK15} one can expect the validity of asymptotic results analogous to those of Section~\ref{s:asympt}. If the channel profile lacks the rotational symmetry with respect to its axis $\Gamma$, one can also expect additional effects coming from the channel \emph{torsion} which could give rise to an effective repulsive interaction in analogy with the results in \cite{EKK08} and \cite[Sec.~1.7]{EK15}
\item \emph{Local perturbations} of potential channels, coming from variation either of their depth or the width. The question about the existence of a discrete spectrum is easy to answer if such a perturbation is sign-definite, the mentioned result of \cite{EV97} suggests that in general it may be harder.
\item Potential channels of a more complicated geometry, in first place \emph{branched ones} where the `axis' $\Gamma$ is a metric graph. If the `ends' of such a channel system have the same profile being (asymptotically) straight, the essential spectrum threshold is determined by that of a single semi-infinite channel; the existence question can be sometimes answered using the minimax principle, provided there is a part of the potential in the form of a suitable bent or polygonal channel. Of course, to have the problem well defined one must specify the potential in the vicinity of the graph vertices because the spectrum would depend on it. The simplest example to consider is the flat-bottom guide when the potential equals $-V_0$ at the points the distance of which from $\Gamma$ is less than $a$ and zero otherwise.
\item In addition to the discrete spectrum existence, one is interested in the \emph{number of eigenvalues} and their other properties in dependence on the system geometry. Of particular interest are the \emph{weakly bound states} corresponding to mild geometric perturbations. One such question is whether the gap between the ground state and the continuum in slightly bent potential channels would be proportional to the fourth power of the bending angle as it is the case for hard-wall tubes \cite[Thm.~6.3]{EK15} and leaky wires \cite{EKo15}.
\item Various \emph{spectral optimization} problems come to mind. For instance, if the potential channel is finite and $\Gamma$ is a loop of a fixed length, one can ask about the shape which makes the principal eigenvalue maximal; in analogy with \cite[Prop.~3.2.1 and Thm.~10.6]{EK15} we conjecture that a sharp maximum is reached by a circular shape.
\item The discrete spectrum is not the only interesting feature of such system, another one concerns the \emph{scattering} in a bent or locally perturbed potential channel; the first task here is to establish the existence and completeness of wave operators as it is done for hard-wall tubes and leaky wires, cf. \cite[Sec.~2.1, Thm.~10.4]{EK15} and \cite{Di20, EKo19}. Furthermore, if the profile of the channel is narrow and sufficiently deep so that $h_V$ has more than one bound state, one can expect the presence of \emph{resonances} near the corresponding higher thresholds in analogy with \cite{Ne97} and \cite[Thm.~2.6]{EK15}.
\item Another extension to three dimensions concerns \emph{potential layers}, that is potentials of a fixed transverse profile built over an infinite surface $\Sigma$ in $\R^3$, different from a plane but asymptotically flat in a suitable sense. Here we know that the discrete spectrum of the hard-wall layer is nonempty if the total Gauss curvature of $\Sigma$ is non-positive \cite[Thm.~4.1]{EK15}, in the opposite case we have partial results only \cite[Sec.~4.2]{EK15}. In the spirit of Proposition~\ref{prop:hardwall} we can establish the discrete spectrum existence for potential layers with the profile deep enough, while in the regime different from the asymptotic one, the question is open.
\item A substantial difference between the tubes and layers is that in the latter case the spectrum depends on the \emph{global} geometry of the interaction support. A nice illustration is provided by \emph{conical} hard-wall layers \cite{ET10} and leaky surfaces \cite{BEL14} which both have an infinite discrete spectrum accumulating logarithmically at the continuum threshold, even in the situations when the generating surfaces lacks a rotational symmetry \cite{OP18}. We noted in the introduction a recent result \cite{EKP20} showing that the spectrum of a cylindrical potential layer\footnote{The channel profile in this paper is supposed to have mirror symmetry, while the other assumptions are less restrictive than here, in particular, the support of $V$ may not be compact and the potential need not be sign-definite.} has the same behavior. One can conjecture that also potential layers with a different geometry might behave as their hard-wall counterparts, such as treated in \cite{EL20}, irrespective of the channel profile.
\item Another subject of interest concerns the influence of external fields. If a two-dimensional hard-wall tube is exposed to a local magnetic field perpendicular to the plane, the corresponding Hamiltonian satisfies a Hardy-type inequality \cite{EK05} that prevents the existence of weakly bound states; the question is what such a perturbation would do with a soft waveguide. On the other hand, if the magnetic field is homogeneous, both the essential and the discrete spectrum change; recall that the question whether curvature-induced bound states in hard-wall waveguides would survive a strong magnetic field remains open. At the same time, a homogeneous field induces a transport along the potential channel analogous to that studied, for instance, in \cite{FGW00, HS08}; the question is the stability of the corresponding edge currents with respect to various perturbations.
\item In \emph{periodic waveguides} one is interested primarily in the absolute continuity of the spectrum and the existence of spectral gaps. As for the latter, having again Proposition~\ref{prop:hardwall} in mind, one can conjecture that the gaps would exist for profiles deep an narrow enough. The (global) absolute continuity is in the hard-wall case proved in the two-dimensional situation \cite{SW02}, however, even for periodic leaky wires it remains an open problem.
\item All the above listed problems concerned the one-particle situation. A completely new area opens when we consider a system of \emph{many particles} interacting mutually, for instance, due the charges they carry, confined in a soft waveguide. In the hard-wall case, for instance, a condition is known \cite[Thm.~3.7]{EK15} under which more than one particle cannot be bound. In the soft case one can expect that the same would happen for a large enough particle charge, but the specific form of such a condition remains to be found.
\end{enumerate}

This list is by no means complete but we prefer to stop here with the hope that we convinced the reader that this area offers a large number of interesting questions.

\subsection*{Acknowledgements}

The author is grateful to the referees for useful comments. The research was supported in part by the European Union within the project CZ.02.1.01/0.0/0.0/16 019/0000778.



\section*{References}

\begin{thebibliography}{99}
\setlength{\itemsep}{2pt}

\bibitem[AS72]{AS72}
M.~Abramowicz, I.A.~Stegun: \emph{Tables of Integral, Series, and Products}, 10th printing, National Bureau of Standatds 1972.
\bibitem[BEHL17]{BEHL17}
J.~Behrndt, P.~Exner, M.~Holzmann, V.~Lotoreichik: Approximation of Schr\"odinger operators with $\delta$-interactions supported on hypersurfaces, \emph{Math. Nachr.} \textbf{290} (2017), 1215--1248.
\bibitem[BEL14]{BEL14}
J. Behrndt, P.~Exner, V. Lotoreichik: Schr\"odinger operators with $\delta$-interactions supported on conical surfaces, \emph{J. Phys. A: Math. Theor.} \textbf{47} (2014), 355202.
\bibitem[BGRS97]{BGRS97}
W.~Bulla, F.~Gesztesy, W.~Renger, B.~Simon: Weakly coupled bound states in quantum waveguides, \emph{Proc. Amer. Math. Soc.} \textbf{127} (1997), 1487-1495.
\bibitem[DH93]{DH93}
M.~Dauge, B. Helffer: Eigenvalue variation. I. Neumann problem for Sturm-Liouville operators, \emph{J. Diff. Eqs} \textbf{104} (1993), 243--262.
\bibitem[DP98]{DP98}
E.B.~Davies, L.~Parnovski: Trapped modes in acoustic waveguides, \emph{Quart. J. Mech. Appl. Math.} \textbf{51} (1998), 477-492.
\bibitem[DK05]{DK05}
M.~Demuth, M.~Krishna: \emph{Determining Spectra in Quantum Theory}, Birkh\"auser, Boston 2005.
\bibitem[Di20]{Di20}
J.~Dittrich: Scattering of particles bounded to an infinite planar curve, \emph{Rev. Math. Phys.} \textbf{32} (2020), to appear; \texttt{arXiv:1912.03958 [math-ph]}
\bibitem[EKP18]{EKP18}
S.~Egger, J.~Kerner, K.~Pankrashkin: Bound states of a pair of particles on the half-line with a general interaction potential, \emph{J. Spect. Theory}, to appear; \texttt{arXiv:1812.06500 [math-ph]}
\bibitem[EKP20]{EKP20}
S.~Egger, J.~Kerner, K.~Pankrashkin: Discrete spectrum of Schr\"odinger operators with potentials concentrated near conical surfaces, \emph{Lett. Math. Phys.} \textbf{110} (2020), 945--968.
\bibitem[EK05]{EK05}
T.~Ekholm, H.~Kova\v r\'{\i}k: Stability of the magnetic Schr\"odinger operator in a waveguide, \emph{ Comm.~PDE} \textbf{30} (2005), 539-565.
\bibitem[EKK08]{EKK08}
T.~Ekholm, H.~Kova\v r\'{\i}k, D.~Krej\v ci\v{r}\'{\i}k: A Hardy inequality in twisted waveguides, \emph{Arch. Rat. Mech. Anal.} \textbf{188} (2) (2008), 245-264.
\bibitem[EI01]{EI01}
P.~Exner, T.~Ichinose: Geometrically induced spectrum in curved leaky wires, \emph{J. Phys. A: Math. Gen.} \textbf{34} (2001), 1439--1450.
\bibitem[EKo15]{EKo15}
P.~Exner, S.~Kondej: Gap asymptotics in a weakly bent leaky quantum wire, \emph{J. Phys. A: Math. Theor.} \textbf{48} (2015), 495301.
\bibitem[EKo19]{EKo19}
P.~Exner, S.~Kondej: Scattering on leaky wires in dimension three, in \emph{Analysis and Operator Theory -- In Honor of Tosio Kato's 100th Birthday} (Th. Rassias and V. Zagrebnov, eds.), Springer Optimization and Its Applications, vol.~146, Cham 2019; pp.~81--91.
\bibitem[EK15]{EK15}
P.~Exner, H.~Kova\v{r}\'{\i}k: \emph{Quantum Waveguides}, Springer International, Heidelberg 2015.
\bibitem[EL20]{EL20}
P.~Exner, V.~Lotoreichik: Spectral asymptotics of the Dirichlet Laplacian on a generalized parabolic layer, \emph{Int. Eqs Oper. Theory} \textbf{92} (2020), 15.
\bibitem[ET10]{ET10}
P.~Exner, M.~Tater: Spectrum of Dirichlet Laplacian in a conical layer, \emph{J. Phys. A: Math. Theor.} \textbf{43} (2010), 474023.
\bibitem[EV97]{EV97}
P.~Exner, S.A.~Vugalter: Bounds states in a locally deformed wave\-guide: the critical case, \emph{Lett. Math. Phys.} \textbf{39} (1997), 59--68.
\bibitem[FGW00]{FGW00}
J. Fr\"ohlich, G.M. Graf, J. Walcher: On the extended nature of edge states of quantum Hall Hamiltonians, \emph{Ann. Henri Poincar\'e} \textbf{1} (2000), 405--442.
\bibitem[GR07]{GR07}
I.S.~Gradshtein, I.M.~Ryzhik: \emph{Tables of Integral, Series, and Products}, 7th edition, Academic Press, New York 2007.
\bibitem[HS08]{HS08}
P.~Hislop, E.~Soccorsi: Edge currents for quantum Hall systems. II. Two-edge, bounded and unbounded geometries, {\em Ann. Henri Poincar\'e} {\bf 9} (2008), 1141--1175.
\bibitem[LLM86]{LLM86}
F.~Lenz, J.T.~Londergan, R.J.~Moniz, R.~Rosenfelder, M.~Stingl, K.~Yazaki: Quark confinement and hadronic interactions, \emph{ Ann. Phys.} \textbf{170} (1986), 65--254.
\bibitem[LCM99]{LCM99}
J.T.~Londergan, J.P.~Carini, D.P.~Murdock: \emph{Binding and Scattering in Two-Dimensional Systems. Applications to Quantum
Wires, Waveguides and Photonic Crystals}, Springer LNP~m60, Berlin 1999.
\bibitem[Ne97]{Ne97}
L.~Nedelec:  Sur les r\'esonances de l'op\'erateur de Dirichlet dans un tube, \emph{Comm. PDE} \textbf{22} (1997), 143--163.
\bibitem[OP18]{OP18}
T.~Ourmi\`eres-Bonafos, K.~Pankrashkin: Discrete spectrum of interactions concentrated near conical surfaces, \emph{Appl. Anal.} \textbf{97} (2018), 1628--1649.
\bibitem[Si05]{Si05}
B.~Simon: \emph{Functional Integration in Quantum Physics}, 2nd edition, AMS Chelsea, Providence, R.I. 2005.
\bibitem[SW02]{SW02}
A.~Sobolev, J.~Walthoe: Absolute continuity in periodic waveguides, \emph{Proc. London Math. Soc.} \textbf{85} (2002), 717--741.


\end{thebibliography}
\end{document}